 \documentclass[a4paper,12pt]{amsart}
 \usepackage{amsmath,amsthm,amssymb}
 \usepackage{graphicx}
 \usepackage[all]{xy}
 \usepackage{txfonts}
\usepackage{hyperref}
\usepackage{color}
\usepackage[margin=2cm]{geometry}

\theoremstyle{definition}
\newtheorem{thm}{Theorem}[section]
\newtheorem{lem}[thm]{Lemma}
\newtheorem{cor}[thm]{Corollary}
\newtheorem{prop}[thm]{Proposition}

\newtheorem{rem}[thm]{\rm Remark}
\newtheorem{ex}[thm]{\rm Example}
\numberwithin{equation}{section}

\newcommand{\Z}{\mathbb{Z}}

\newcommand{\F}{\mathbb{F}}

\newcommand{\ismax}{\operatorname{ismax}}
\newcommand{\argmax}{\operatorname{argmax}}
\newcommand{\argmin}{\operatorname{argmin}}

\newcommand{\nummax}{\mathrm{nummax}}

\title{Polynomial expressions of $p$-ary auction functions}
\author{Shizuo Kaji}
\address[Shizuo Kaji]{Yamaguchi University, Japan \endgraf
Japan Science and Technology Agency (JST) PRESTO Researcher}
\email{skaji@yamaguchi-u.ac.jp}
\author{Toshiaki Maeno}
\address[Toshiaki Maeno]{Meijo University, Japan}
\email{tmaeno@meijo-u.ac.jp}
\author{Koji Nuida}
\address[Koji Nuida]{National Institute of Advanced Industrial Science and Technology (AIST), Japan \endgraf
Japan Science and Technology Agency (JST) PRESTO Researcher}
\email{k.nuida@aist.go.jp}
\author{Yasuhide Numata}
\thanks{The fourth named author was partially supported by KAKENHI, Grant-in-Aid for Young 
     Scientists (B) JP25800009.}
\address[Yasuhide Numata]{Shinshu University, Japan}
\email{nu@math.shinshu-u.ac.jp}


\date{\today}

\begin{document}
\begin{abstract}
Let $\F_p$ be the finite field of prime order $p$.
For any function $f \colon \F_p{}^n \to \F_p$,
there exists a unique polynomial over $\F_p$ having degree at most $p-1$ with respect to each variable
which coincides with $f$. We call it the minimal polynomial of $f$.
It is in general a non-trivial task to find a concrete expression of the minimal polynomial of a given function, 
which has only been worked out for limited classes of functions in the literature.
In this paper, we study minimal polynomial expressions of several functions that are closely related to some practically important procedures such as auction and voting.
\end{abstract}

\maketitle

\section{Introduction}
Let $p$ be a prime and $\F_p$ the finite field of order $p$.
It is well-known that any function $f \colon \F_p{}^n \to \F_p$ {can be} expressed as a polynomial $P(x_1,\dots,x_n)$
with coefficients in $\F_p$, and such a polynomial is unique if its degree with respect to each variable is restricted to be at most $p-1$; we call the unique polynomial $P$ the \emph{minimal polynomial} of the function $f$.
In theory, it is easy by Fermat's Little Theorem to see that the polynomial $P$ is given by 
$\sum_{(a_1,\dots,a_n)\in \F_p{}^n} f(a_1,\dots,a_n) \delta_{a_1}(x_1) \cdots \delta_{a_n}(x_n)$, 
where $\delta_{a_i}(x_i)=1-(x_i-a_i)^{p-1}$  is the minimal polynomial for the Kronecker delta.
This expression, however, has two shortcomings; it relies on the (often implicit) values $f(a_1,\dots,a_n)$ of the function,
and it in general contains many redundant terms to be cancelled out.
As a result, it remains a non-trivial task to obtain an \emph{explicit} and concise minimal polynomial expression for a given concrete function $f$.
For example, Sturtivant and Frandsen \cite[Theorems 9.1(a) and 11.2]{SF93} showed that the carry function in multiplication of $p$-ary integers is expressed by using number-theoretic objects such as the Bernoulli numbers and Wilson's quotient (see also \cite{carry} for a different approach to the result and an expression of the carry function in the case of addition of $p$-ary integers).
As this previous result suggests,
the problem of computing minimal polynomial expressions of certain functions 
can lead to interesting theoretical results connecting different fields of mathematics.

On the other hand, this problem has potential applications in cryptography as well. 
There was recently a breakthrough in the area of cryptography, namely the discovery of \emph{fully homomorphic encryption} 
(see \cite{Gen14,Sil13} for survey).
One can compute \emph{in an encrypted form} both addition and multiplication over the two-element field $\F_2$ (see \cite{Gen09}, etc.{}) 
and over even larger finite prime fields $\F_p$ for $p > 2$ (see \cite{NK15}).
{It follows that one can compute any function provided the function is explicitly written as a polynomial over $\F_p$.
For example, a recent work \cite{CKK15} on practical cryptographic systems based on fully homomorphic encryption relies on a recursive polynomial expression of the comparison function for two binary integers.
To develop such practical systems, \lq\lq efficient'' polynomial expressions of various functions are useful, and in particular, the minimal degree condition is important since encrypted multiplication is in general computationally much more expensive than encrypted addition.}

In this paper, we study minimal polynomial expressions of a certain kind of functions specified below.
They are relevant to some practical procedures such as auction and voting.
We chiefly discuss the $\max$ function that takes an element of $\F_p{}^n$ as input and returns the largest value among them,
 and the $\argmax$ function that returns the least index of the largest component(s) in the input vector in $\F_p{}^n$.
Here we clarify that, the finite field $\F_p$ is naturally identified with the subset $\{0,1,\dots,p-1\}$ of integers, and comparison of elements (e.g., in the function $\max$) is performed in the latter, 
while addition and multiplication are done in the former.
The output of $\argmax$ is an integer that may exceed the range of the field $\F_p$ when $n\ge p$.
To handle this, we introduce an $\F_p$-valued function $\argmax^{(r)}$ that returns the $r$-th digit of the $p$-ary expansion of $\argmax$.

In \S \ref{sec:notation}, we define and give the minimal polynomial for the
\lq\lq low-pass filtering function'' $L_t(x)$
and the Kronecker delta function $\delta_t(x)$, which are used as building blocks in the later sections.
In \S \ref{sec:max}, we give a minimal polynomial expression of the function $\max$  in terms of 
$L_t(x)$ and $\delta_t(x)$.
However, these general expressions contain many terms. We derive more concise forms for $p=2$ and $3$
(Corollary \ref{psi-p-2} and Proposition \ref{psi-p-3}).
A duality between $\max$ and $\min$ allows us to deduce corresponding formulae for $\min$.
\S \ref{sec:argmax} is devoted to the study of the $\argmax$ function.
First, we provide a way of reducing the computation for $\argmax^{(r)}$ with any $r \geq 0$ to the computation for $\argmax^{(0)}$ by utilising the result on the function $\max$.
We also provide a recursive formula for $\argmax^{(r)}$ with respect to the input length.
They are used to derive minimal polynomial expressions of $\argmax^{(r)}$ when $p = 2$
and of $\argmax^{(0)}$ for $p = 3$ and $n = 3$ 
(Propositions \ref{prop:argmax_p=2_first}, \ref{prop:argmax_p=2_second}, and \ref{prop:argmax_p=3_n=3}).
The recursive formula for $\argmax^{(r)}$ relies on the (minimal) polynomial expression of $\argmax^{(0)}$ with input length of two.
We give a minimal polynomial expression of $\argmax^{(0)}$ for $n = 2$ and any $p$ in \S \ref{sec:two-variables},
 which also yields a minimal polynomial expression of $\max$ with $n = 2$ (Proposition \ref{prop:argmax_n=2} and Theorem \ref{prop:max_n=2}).

In \S \ref{sec:other_functions}, we introduce and study two more functions that are also relevant to our problem.
We recall that the definition of $\argmax$ enforces the function to always output the first index when there are ties in the input vectors; this then loses the information on the other largest components of the input.
To remedy this situation, 
we introduce the function \lq\lq $\ismax(y;x)$'' that returns if the maximum value among the components of the input vector $x$ is equal to the other input value $y \in \F_p$, and \lq\lq $\nummax(x)$'' that returns the number of inputs which attain the tied maximum.
Then, similarly to the cases of $\max$ and $\argmax$, we provide a general formula for the
minimal polynomial expression of $\ismax$ and $\nummax$ in terms of the low-pass filtering functions and the Kronecker delta functions, and also compute concise forms of minimal polynomial expressions of $\ismax$ for $p = 2$ and $3$, and of $\nummax$ for $p=2$.

We conclude with a possible extension of our result to a multi-digit setting in \S \ref{sec:multi-digit}.

\subsection*{Acknowledgement}
The authors would like to thank Takuro Abe 
for fruitful discussions.

\section{Notation and Basic functions}
\label{sec:notation}
In this section, we fix some notations used throughout the paper.
A vector $x$ of length $n$ over the field $\F_p$ is denoted by $(x_0,x_1,\ldots,x_{n-1})$.
We introduce a linear ordering $<$ on $\F_p$ via 
the natural identification of it with the subset $\{0,1,\ldots,p-1\}$ of $\Z$ (with the usual ordering $<$).
We denote by $e_i(x)$ the $i$-th elementary symmetric polynomial of $x_0,x_1,\ldots,x_{n-1}$
so that $\prod_{i=0}^{n-1} (1+x_i) = \sum_{i=0}^n e_i(x)$.

For a logical formula $P$ with free variable $x$, we define its \emph{truth function} by
\[
\chi_P(x) = \begin{cases} 1 & (P(x) \text{ is true}) \\ 0 & (\text{otherwise}) \end{cases}
\]
which is often abbreviated as $\chi_P(x)=\chi(P)$.
We frequently use the same symbol for a function and its polynomial expression.
\begin{ex}\label{ex:delta}
For $t \in \F_p$, the minimal polynomial for the \emph{delta function} $\delta_t(x) = \chi(x=t)$ is given by
\[
\delta_t(x) = 1-(x-t)^{p-1}=-\prod_{i=1}^{p-1}(x-t+i) \enspace,
\]
which follows from Fermat's Little Theorem.
Similarly, the minimal polynomial for the \emph{low-pass function} $L_t(x) = \chi(x<t)$ is given by
\[
L_t(x) = \sum_{0 \le k < t} \delta_k(x) = \sum_{0 \le k < t}\left( 1-(x-k)^{p-1} \right) \enspace.
\]
\end{ex}

For an integer $k \geq 0$, its $r$-th digit in the $p$-ary expansion is denoted by $k^{(r)}$;
that is, $k=\sum_{r=0}^\infty k^{(r)} p^r$ with $k^{(r)} \in \{0,1,\ldots,p-1\}$ for each $r$.

\section{Polynomial expressions of the $\max$ and the $\min$ functions}
\label{sec:max}
For a vector $x = (x_0,x_1,\ldots,x_{n-1}) \in \F_p{}^n$, let $\max(x)$ (respectively, $\min(x)$) denote the maximum (respectively, minimum) among the $n$ values $x_0,x_1,\ldots,x_{n-1}$.

Using the functions in Example \ref{ex:delta}, we immediately obtain the minimal polynomial of $\max$.
\begin{prop}\label{prop:max-general}
The minimal polynomial of $\max$ is given by
\begin{eqnarray*}
\max(x) &=& \sum_{1 \le t \le p-1} \chi( x_i\ge t \mbox{ for some } i ) = 
\sum_{1 \le t \le p-1} \left( 1- \prod_{i=0}^{n-1} L_t(x_i) \right) \enspace.
\end{eqnarray*}
\end{prop}
In particular, when $p=2$ this simplifies:
\begin{cor}
\label{psi-p-2}
The minimal polynomial of $\max(x)$ for $p=2$ is given by
\[
\max(x)=\prod_{i=0}^{n-1} (1+x_i)-1
=\sum_{i=1}^{n}e_{i}(x) \enspace.
\]
\end{cor}

However when $p>2$, the expression in Proposition \ref{prop:max-general} consists of a lot of terms.
We now compute a more concise expression for $p=3$.

First we note that $\max(x)+1=0$ if $x_i=p-1$ for some $x_i$.
This implies that the minimal polynomial of $\max(x) + 1$ has $1 + x_i$ as a factor for every $i$. Therefore, we have
\[
\max(x) = f_n(x)\prod_{i=0}^{n-1} (1+x_i) - 1 = f_n(x)\sum_{i=0}^{n}e_i(x) - 1
\]
for some polynomial $f_n(x)$ in which each variable $x_i$ has degree at most $p-2$.
In particular, this observation yields another proof of Corollary \ref{psi-p-2} (where $p = 2$).
For the case $p = 3$, we have the following result:

\begin{prop}
\label{psi-p-3}
When $p=3$, a minimal polynomial expression for $\max(x)$ is given by:
\[
\max(x)
=\sum_{i=0}^{\lfloor n/2 \rfloor} e_{2i}(x) \sum_{i=0}^{n}e_{i}(x)  -1 
\enspace.
\]
\end{prop}
\begin{proof}
Denote the right hand side by $P(x)$.
As the minimality condition on the degree is satisfied for $P(x)$, 
it suffices to verify $\max(x)=P(x)$ for any $x\in \F_p{}^n$.
When $\max(x)=2$, there exists $i$ such that $x_i=2$. This implies $\prod_{i=0}^{n-1} (1+x_i)=\sum_{i=0}^{n}e_{i}(x)=0$
and $P(x)=-1=2$.
Notice that 
\[
\sum_{i=0}^{\lfloor n/2 \rfloor} e_{2i}(x) \sum_{i=0}^{n}e_{i}(x)=2\left(\prod_{i=0}^{n-1} (1+x_i)^2 + \prod_{i=0}^{n-1} (1-x_i^2) \right)
\]
by the definition of $e_i(x)$ and the fact $2^{-1}=2$ in $\F_3$.
When $\max(x)=1$, as $\prod_{i=0}^{n-1} (1+x_i)^2=1$ and 
$\prod_{i=0}^{n-1} (1-x_i^2)=0$, we have $P(x)=2(1+0) - 1 = 1$.
When $\max(x)=0$, as $\prod_{i=0}^{n-1} (1+x_i)^2=1$ and 
$\prod_{i=0}^{n-1} (1-x_i^2)=1$, we have $P(x)=2(1+1) - 1 = 0$.
\end{proof}

To obtain a minimal polynomial expression for $\min$, we exploit a duality between $\max$ and $\min$.
Define an \emph{involution} on $\F_p$ by $\bar{x} = p-1-x$ and extend it coordinate-wisely on $\F_p{}^n$.
Then, we have $\overline{\min(x)} = \max(\bar{x})$ for any $x\in \F_p{}^n$. 
Thus, a minimal polynomial expression for $\max$ converts to one of $\min$ and vice versa.
For example, Corollary \ref{psi-p-2} and Proposition \ref{psi-p-3} imply the following:
\begin{cor}
When $p=2$, a minimal polynomial expression for $\min$ is given by
\[
\min(x) = \prod_{i=0}^{n-1} x_i = e_n(x) \enspace.
\]
When $p=3$, a minimal polynomial expression for $\min$ is given by
\[
\min(x) = \prod_{i=0}^{n-1} x_i^2 + \prod_{i=0}^{n-1} x_i(1-x_i) = e_n \left( 1+\sum_{i=1}^n (-1)^i e_i+e_n \right) \enspace.
\]
\end{cor}

For the next case of $p = 5$, minimal polynomial expressions of $\max(x)$ for small values of $n$ in terms of elementary symmetric polynomials can be determined by direct calculation:

\begin{ex}
When $p=5$, the following are minimal polynomial expressions.
\begin{itemize}
\item $\max(x_0,x_1)=(1+e_1+e_2)(1+2e_1^2e_2+4e_1e_2+e_2)-1$
\item $\max(x_0,x_1,x_2)=(1+e_1+e_2+e_3)(
1+2e_1^2e_2+e_1e_2e_3+2e_1 e_3^2+e_2^2 e_3+2e_2e_3^2+4e_1e_2+3e_1e_3+e_2e_3+3e_3^2+e_2)-1$
\end{itemize}
\end{ex}

However, it seems to be difficult to obtain a general formula (such as Proposition\ref{psi-p-3}) for $p \ge 5$.
The function $\max(x)$ with $n = 2$ for any $p$ will be revisited in \S \ref{sec:argmax}.

\begin{rem}
The function $\max(x)$ is a symmetric function (in variables $x_0,x_1,\ldots,x_{n-1}$), and satisfies $\max(x,0)=\max(x)$ and an \lq\lq associativity'' in the following sense:
\[
\max(x_0,x_1,\ldots,x_{n-1},x_n)=\max(\max(x_0,\ldots,x_{n-1}),x_n)=\max(x_0,\max(x_1,\ldots,x_n)) \enspace.
\]
By using this property recursively, a minimal polynomial expression of the $\max$ function with two variables (i.e., for the case $n = 2$) yields a polynomial expression of $\max$ with any number of variables (i.e., for any $n$).
However, the polynomial thus obtained is \emph{not} the minimal polynomial in general.
\end{rem}

\section{Polynomial expressions of the $\argmax$ function}
\label{sec:argmax}

Let $\argmax(x)$ be the least integer $i$ such that $x_i=\max(x)$. 
Note that $\argmax(x)$ takes values in $\{0,1,\ldots,n-1\}$ so
we define for $r \geq 0$
\[
\argmax^{(r)} \colon \F_p^n \to \F_p \enspace,\quad \argmax^{(r)}(x) = \argmax(x)^{(r)},
\]
where $\argmax^{(r)}(x)$ is the $r$-th digit in the $p$-ary expansion of $\argmax(x)$.

Again using the functions in Example \ref{ex:delta}, we immediately obtain the minimal polynomial of $\argmax^{(r)}$.
\begin{prop}
The minimal polynomial for $\argmax^{(r)}$ is given by
\begin{eqnarray*}
\argmax^{(r)}(x)&=& \sum_{i=0}^{n-1}  i^{(r)} \cdot \chi(\argmax(x)=i)  \\
&=&  \sum_{i=0}^{n-1}  i^{(r)} \left( \sum_{0 \le t \le p-1} \left( 
 \delta_t(x_i) \prod_{0 \le j<i} L_t (x_j) \prod_{i<k \le n-1} L_{t+1} (x_k) \right)
  \right) \enspace.
\end{eqnarray*}
\end{prop}
\begin{rem}
Let $\argmin(x)$ be the function which returns the least index $i$ with $\min(x) = x_i$.
A minimal polynomial expression of $\argmin$ 
 is obtained from one of $\argmax$ via the duality $\argmin(x)=\argmax(\bar{x})$ 
 similarly to the case of $\min$ discussed in \S \ref{sec:max}.
\end{rem}

Observe by definition of the function $\argmax^{(r)}$ that
\begin{equation}
\label{eq:argmax_recurrence}
\argmax^{(r)}(x) = \argmax^{(0)}(\max(x_{0},x_{1},\ldots,x_{p^{r}-1}),
\ldots,
\max(x_{i\cdot p^r},x_{i\cdot p^r+1},\ldots,x_{(i+1)\cdot p^{r}-1}),
\ldots) \enspace,
\end{equation}
\begin{equation}
\label{eq:argmax_recursive}
\begin{split}
\argmax^{(r)}(x_0,\ldots,x_{n-1},x_{n})
=& \argmax^{(r)}(x_0,\ldots,x_{n-1}) \cdot \bigl( 1-\argmax^{(0)}(\max(x_0,\ldots,x_{n-1}),x_{n}) \bigr) \\
& + n^{(r)} \cdot \argmax^{(0)}(\max(x_0,\ldots,x_{n-1}),x_{n}) \enspace.
\end{split}
\end{equation}
The second equation follows from
\[
\argmax(x_0,\ldots,x_{n-1},x_{n})=
\begin{cases}
\argmax(x_0,\ldots,x_{n-1}) & \mbox{if $\max(x_0,\ldots,x_{n-1}) \ge x_{n}$,} \\
n & \mbox{if $\max(x_0,\ldots,x_{n-1}) < x_{n}$.}
\end{cases}
\]

These formulae yield a (in general, not minimal) polynomial expression of $\argmax^{(r)}(x)$ from
 those of $\argmax^{(0)}(x)$ with $n = 2$ and $\max(x)$.
\subsection{The case $p = 2$}
\label{subsec:argmax__p=2}
When $p=2$, we can derive a minimal polynomial expression of $\argmax^{(r)}$.
\begin{prop}
\label{prop:argmax_p=2_first}
When $p=2$ and $k,r \ge 0$, the following are minimal polynomial expressions:
\begin{eqnarray*}
\argmax^{(r)}(x_0, x_1,\ldots,x_{(2k+2)2^r-1})&=& \argmax^{(r)}(x_0, x_1,\ldots,x_{(2k+1)2^r-1})\\
&=&
\sum_{i=0}^k (1 + x_0)(1 + x_1)\cdots(1 + x_{(2i+1)2^r-1})  \left(\prod_{j=(2i+1)2^r}^{(2i+2)2^r-1} (1+x_{j}) -1\right)\\
&=&
\sum_{i=1}^{2k+2} (1 + x_0)(1 + x_1)\cdots(1 + x_{i \cdot 2^r-1}) \enspace, \\
\argmax^{(r)}(x_0,x_1,\ldots,x_{n-1}) 
&=& \argmax^{(r)}(x_0,x_1,\ldots,x_{n-1},0) \enspace.
\end{eqnarray*}
\end{prop}
\begin{proof}
Notice that $\argmax^{(0)}(x) = 1$ if and only if there is an odd index $i$ satisfying that $x_j = 0$ for every $j < i$ and $x_i = 1$.
So we have
\[
\argmax^{(0)}(x_0, x_1,\ldots,x_{2k+2})
=\argmax^{(0)} (x_0,x_1,\ldots ,x_{2k+1})
=\sum_{i=0}^k (1 + x_0)(1 + x_1)\cdots(1 + x_{2i})x_{2i+1} \enspace.
\]
Combining this with \eqref{eq:argmax_recurrence} and Proposition \ref{psi-p-2}, we obtain the first formula
(note that the characteristic is now $p = 2$).
The second formula follows from the fact $\argmax(x,0) = \argmax(x)$.
\end{proof}

We can also use \eqref{eq:argmax_recursive} to give another formula:
\begin{prop}
\label{prop:argmax_p=2_second}
When $p = 2$ and $r \geq 0$,
a minimal polynomial expression of $\argmax^{(r)}(x)$ is given as follows:\\
$(1)$ $\argmax^{(r)}(x_0) = 0$, \\
$(2)$ If $2^{r+1}k+2^r \le n < 2^{r+1}(k+1)$ for an integer $k \ge 0$ (i.e., $n^{(r)} = 1$), then we have
\[
\argmax^{(r)}(x_0,x_1,\ldots,x_n)
= \sum_{j=0}^{k}\sum_{i=0}^{\min (2^r-1,n-2^{r+1}j-2^r)}
(1+x_0)(1+x_1)\cdots (1+x_{2^{r+1}j+2^r+i-1}) x_{2^{r+1}j+2^r+i} \enspace,
\]
$(3)$ If $2^{r+1}k \le n < 2^{r+1}k+2^r$ for an integer $k \ge 0$ (i.e., $n^{(r)} = 0$), then we have
\[
\argmax^{(r)}(x_0,x_1,\ldots,x_n) = \argmax^{(r)}(x_0,x_1,\ldots,x_{n-1}) \enspace.
\]
\end{prop}
\begin{proof}
Let $f(x) = f(x_0,\ldots,x_n)$ denote the right-hand side of the claimed equality in the statement (we define $f(x) = 0$ when $n = 0$).
First we note that, in the present case $p = 2$ we have
\[
\argmax^{(r)}(x_0,\ldots,x_{n-1}) \argmax^{(0)}(\max(x_0,\ldots,x_{n-1}),x_n) = 0 \enspace,
\]
since the only possibility to satisfy $\argmax^{(0)}(\max(x_0,\ldots,x_{n-1}),x_n) = 1$ is that $x_i = 0$ for every $i < n$ and $x_n = 1$, which then implies $\argmax(x_0,\ldots,x_{n-1}) = 0$.
Therefore, the recursive formula \eqref{eq:argmax_recursive} now becomes
\[
\argmax^{(r)}(x_0,\ldots,x_{n-1},x_n)
= \argmax^{(r)}(x_0,\ldots,x_{n-1}) + n^{(r)} \cdot \argmax^{(0)}(\max(x_0,\ldots,x_{n-1}),x_n) \enspace.
\]
It then suffices to show that $f(x)$ instead of $\argmax^{(r)}(x)$ also satisfies the same recursive formula.
This is obvious when $n$ satisfies the condition for the second case in the statement.

From now on, we focus on the other case where $n$ satisfies the condition for the first case in the statement.
Since direct computation shows $\argmax^{(0)}(x_0,x_1)=(x_0+1)x_1$,
by Proposition \ref{psi-p-2} we have
\[
\argmax^{(0)}(\max(x_0,\ldots,x_{n-1}),x_n)
= (\max(x_0,\ldots,x_{n-1}) + 1) x_n
= (1 + x_0)(1 + x_1) \cdots (1 + x_{n-1}) x_n \enspace,
\]
therefore the recursive formula now becomes
\[
\argmax^{(r)}(x_0,\ldots,x_{n-1},x_n)
= \argmax^{(r)}(x_0,\ldots,x_{n-1}) + (1 + x_0)(1 + x_1) \cdots (1 + x_{n-1}) x_n \enspace.
\]
If $n \neq 2^{r+1}k + 2^r$, then $n-1$ also satisfies the same condition as $n$ with the same integer $k$, and now we indeed have $f(x_0,\ldots,x_{n-1},x_n) = f(x_0,\ldots,x_{n-1}) + (1 + x_0)(1 + x_1) \cdots (1 + x_{n-1}) x_n$ by the definition of $f$ (note that $\min (2^r-1,n-2^{r+1}k-2^r) = n-2^{r+1}k-2^r$ in this case).
On the other hand, if $n = 2^{r+1}k + 2^r$, then we have $f(x_0,\ldots,x_{n-1}) = f(x_0,\ldots,x_{2^{r+1}k - 1})$ by the definition of $f$ for the second case in the statement, while $2^{r+1}k - 1$ satisfies the condition for the first case in the statement with $k-1$ playing the role of $k$.
This implies that $f(x_0,\ldots,x_{n-1},x_n) = f(x_0,\ldots,x_{n-1}) + (1 + x_0)(1 + x_1) \cdots (1 + x_{n-1}) x_n$ also holds in this case by the definition of $f$ (note that now $\min (2^r-1,n-2^{r+1}k-2^r) = n-2^{r+1}k-2^r = 0$).
Hence $f$ satisfies the desired recursive formula in any case, completing the proof.
\end{proof}

For Proposition \ref{prop:argmax_p=2_second}, by noting that $x_{2^{r+1}j + 2^r + i} = (1 + x_{2^{r+1}j + 2^r + i}) - 1$ and now the characteristic is $p = 2$, the formula given there can be rewritten as follows.
\begin{cor}
Let $S(r,n)$ be the set of integers defined by
\[
\begin{split}
S(r,n) &:= \{ 2^{r+1}k +2^r -1\; | \; 0\leq k < 2^{-r-1}(n+1-2^r) \} \\
&\quad \cup \{ 2^{r+1}k -1 \; | \; 1\leq k < 2^{-r-1}(n+1-2^r) \} \\
&\quad \cup \{ \min (n, 2^{r+1}(\lfloor (n-2^r)/2^{r+1} \rfloor +1)-1) \} \enspace.
\end{split}
\]
A minimal polynomial expression of $\argmax^{(r)}$ is given by
\[
\argmax^{(r)} (x_0,x_1,\ldots,x_n) = \sum_{i \in S(r,n)} (1+x_0)(1+x_1)\cdots (1+x_i) \enspace.
\]
\end{cor}

\subsection{The case $p = 3$}
\label{subsec:argmax__p=3}
When $p=3$, a direct computation shows
\[
\argmax^{(0)}(x_0,x_1)= x_1(1+x_0)(x_1-x_0).
\]
Combining this with the following formula from Proposition \ref{psi-p-3}
\[
\max(x_0,\ldots,x_n)= \prod_{i=0}^{n} (1+x_i) \cdot \Big( 1+\sum_{i \ge 1} e_{2i}(x_0,\ldots,x_n) \Big) -1 \enspace,
\]
we obtain by \eqref{eq:argmax_recursive}:
\[
\begin{split}
&\argmax^{(r)}(x_0,\ldots,x_n,x_{n+1}) \\
&= \argmax^{(r)}(x_0,\ldots,x_n) \left( 1-x_{n+1} \prod_{i=0}^{n} (1+x_i) \cdot \Bigl( 1+\sum_{i \ge 1} e_{2i}(x_0,\ldots,x_n) \Bigr) \right. \\
&\hspace*{140pt} \left. \cdot \biggl( 1+x_{n+1} - \prod_{i=0}^{n} (1+x_i) \cdot \Bigl( 1+\sum_{i \ge 1} e_{2i}(x_0,\ldots,x_n) \Bigr) \biggr) \right) \\
&\quad + (n+1)^{(r)} \cdot \left( x_{n+1} \prod_{i=0}^{n} (1+x_i) \cdot \Bigl( 1+\sum_{i \ge 1} e_{2i}(x_0,\ldots,x_n) \Bigr) \right. \\
&\hspace*{80pt} \left. \cdot \biggl( 1+x_{n+1} - \prod_{i=0}^{n} (1+x_i) \cdot \Bigl( 1+\sum_{i \ge 1} e_{2i}(x_0,\ldots,x_n) \Bigr) \biggr) \right) \enspace.
\end{split}
\]
Although this formula does not yield a minimal polynomial expression for $\argmax^{(r)}$ directly, 
we can still compute one at least when $n$ is not too large.
For example, when the input vector has length $3$ (hence it suffices to consider $r = 0$ only), the formula above (with $r = 0$ and $n = 1$) becomes
\[
\begin{split}
&\argmax^{(0)}(x_0,x_1,x_2) \\
&= \argmax^{(0)}(x_0,x_1) \biggl( 1-x_2 (1+x_0)(1+x_1)(1+x_0x_1) \Bigl( 1+x_2 - (1+x_0)(1+x_1)(1+x_0x_1) \Bigr) \biggr) \\
&\quad + 2 \cdot \biggl( x_2 (1+x_0)(1+x_1)(1+x_0x_1) \Bigl( 1+x_2 - (1+x_0)(1+x_1)(1+x_0x_1) \Bigr) \biggr) \\
&= x_1(1+x_0)(x_1-x_0) \\
&\quad - \Bigl( x_1(1+x_0)(x_1-x_0) + 1 \Bigr) \biggl( x_2 (1+x_0)(1+x_1)(1+x_0x_1) \Bigl( 1+x_2 - (1+x_0)(1+x_1)(1+x_0x_1) \Bigr) \biggr) \enspace.
\end{split}
\]
A straightforward expansion of the polynomial in the right-hand side yields
\[
\begin{split}
&\argmax^{(0)}(x_0,x_1,x_2) 
= - x_0x_1 - x_0^2x_1 + x_1^2 + x_0x_1^2 \\
&\ + x_2 \bigl( x_0 + x_0^2 + x_1 + x_0x_1 + x_0^2x_1 - x_0^4x_1 + x_1^2 - x_0x_1^2 + x_0^3x_1^2 + x_0^5x_1^2 + x_1^3 - x_0^3x_1^3 + x_0^4x_1^3 - x_0^5x_1^3 - x_0^6x_1^3 \\
&\quad + x_1^4 + x_0^2x_1^4 - x_0^4x_1^4 - x_0^5x_1^4 + x_0^6x_1^4 - x_0x_1^5 - x_0^2x_1^5 - x_0^3x_1^5 - x_0^4x_1^5 - x_0^5x_1^5 - x_0^6x_1^5 + x_0^2x_1^6 + x_0^5x_1^6 \bigr) \\
&\ + x_2^2 \bigl({} - 1 - x_0 - x_1 - x_0x_1 + x_0^2x_1 + x_0^3x_1 - x_1^2 + x_0x_1^2 + x_0^2x_1^2 + x_0^4x_1^2 \\
&\quad - x_1^3 + x_0^2x_1^3 + x_0^3x_1^3 + x_0^4x_1^3 - x_0x_1^4 + x_0^2x_1^4 - x_0^3x_1^4 \bigr)
\end{split}
\]
and, by applying the relations $x_0^3 \equiv x_0$ and $x_1^3 \equiv x_1$ several times (where $\equiv$ means equivalence as functions over $\F_p$), we finally obtain:
\begin{prop}\label{prop:argmax_p=3_n=3}
A minimal polynomial expression of $\argmax^{(0)}(x_0,x_1,x_2)$ for $p = 3$ is given by
\[
\begin{split}
\argmax^{(0)}(x_0,x_1,x_2) 
&=2(x_0x_1^2x_2+x_1^2x_2^2+x_1^2x_2+2x_1x_2^2+x_0x_1+2x_0x_2+2x_1^2+x_1x_2+x_2^2)(x_0+1) \enspace.
\end{split}
\]
\end{prop} 
 
\section{Polynomial expressions of the $\max$ and $\argmax$ functions for two variables}
\label{sec:two-variables}
 First we note that
$x_0<x_1$ if and only if $\overline{x_0}$ and $x_1$ add up to an equal or greater integer than $p$ when
considered as integers;
that is, $\argmax^{(0)}(x_0,x_1)$ is equal to the carry by the $p$-ary addition of two single-digit values $\overline{x_0}$ and $x_1$
to the next digit.
A minimal polynomial expression of this carry function, denoted by $\varphi_1$, has been determined in \cite{carry,NK15}:
\begin{lem}
[{\cite{carry,NK15}}]
\label{lem:carry}
For $y_0,y_1 \in \F_p$, we have
\[
\varphi_1(y_0,y_1)
= \sum_{d=1}^{p-1} (-1)^{d} d^{-1} y_0 (y_0-1) \cdots (y_0-d+1) y_1 (y_1-1) \cdots (y_1-(p-d)+1) \enspace,
\]
where the $d^{-1}$ in the right-hand side means the inverse of $d$ as an element of $\F_p$.
\end{lem}

Combining this with $\argmax^{(0)}(x_0,x_1)=\varphi_1(\overline{x_0},x_1)$, we obtain:
\begin{prop}
\label{prop:argmax_n=2}
When $n=2$, a minimal polynomial expression of $\argmax^{(0)}(x_0,x_1)$ is given by
\[
\argmax^{(0)}(x_0,x_1)=\sum_{d=1}^{p-1} d^{-1} (x_0+1) (x_0+2) \cdots (x_0+d) x_1 (x_1-1) \cdots (x_1-(p-d)+1) \enspace.
\]
\end{prop}

\begin{ex}
\label{ex:argmax_n=2}
By using Proposition \ref{prop:argmax_n=2} (or direct calculation), we have the following minimal polynomial expressions of $\argmax^{(0)}(x_0,x_1)$ for small primes.
\begin{itemize}
\item When $p=2$, $\argmax^{(0)}(x_0,x_1) = (x_0+1)x_1$.
\item When $p=3$, $\argmax^{(0)}(x_0,x_1) = -(x_0+1)(x_0-x_1)x_1$.
\item When $p=5$, $\argmax^{(0)}(x_0,x_1) = -(x_0+1)(x_0^2-x_0 x_1+x_0+x_1^2)(x_0-x_1)x_1$.
\item When $p=7$, $\argmax^{(0)}(x_0,x_1) =$ \\
$-(x_0^4+5 x_0^3 x_1+2 x_0^3+3 x_0^2 x_1^2+x_0^2 x_1+4 x_0^2+
5 x_0 x_1^3+6 x_0 x_1^2+3 x_0+x_1^4) (x_0+1) (x_0-x_1) x_1$.
\end{itemize}
\end{ex}

We also have the following relation between $\max$ and $\argmax$ deduced from their definitions:

\begin{lem}
\label{lem:relation_max_argmax}
We have $\max(x)=\sum_{i=0}^{n-1} x_i \cdot \chi(\argmax(x)=i)$.
In particular, we have
\[
\max(x_0,x_1)=x_0 \cdot (1-\argmax(x_0,x_1)) + x_1 \cdot \argmax(x_0,x_1) \enspace.
\]
\end{lem}

A straightforward substitution of the result of Proposition \ref{prop:argmax_n=2} into the right-hand side of Lemma \ref{lem:relation_max_argmax} yields an almost, but not yet minimal, polynomial expression of $\max(x_0,x_1)$.
This expression can be converted to a minimal polynomial expression.

\begin{thm}
\label{prop:max_n=2}
When $p \geq 3$, we have the following minimal polynomial expression of $\max(x_0,x_1)$:
\[
\begin{split}
\max(x_0,x_1)
&= (x_1-x_0) \sum_{d=2}^{p-2} d^{-1} (x_0+1) (x_0+2) \cdots (x_0+d) x_1 (x_1-1) \cdots (x_1-(p-d)+1) \\
&\quad + x_0 + (x_0+1)^2 (1 - (x_1+1)^{p-1}) + (1 - x_0^{p-1}) x_1^2 \enspace.
\end{split}
\]
\end{thm}
\begin{proof}
Throughout the proof, a notation $f \equiv g$ means that $f$ and $g$ define an identical function on $\F_p$.
First, since $p \geq 3$, Proposition \ref{prop:argmax_n=2} implies
\[
\begin{split}
&x_0 \argmax(x_0,x_1) \\
&= x_0 \sum_{d=1}^{p-1} d^{-1} (x_0+1) (x_0+2) \cdots (x_0+d) x_1 (x_1-1) \cdots (x_1-(p-d)+1) \\
&= x_0 \sum_{d=2}^{p-2} d^{-1} (x_0+1) (x_0+2) \cdots (x_0+d) x_1 (x_1-1) \cdots (x_1-(p-d)+1) \\
&\quad + x_0 (x_0+1) x_1 (x_1-1) \cdots (x_1-(p-2)) - x_0 (x_0+1) (x_0+2) \cdots (x_0+p-1) x_1
\end{split}
\]
and we have $x_0 (x_0+1) (x_0+2) \cdots (x_0+p-1) x_1 \equiv 0$ for the last term above.
Similarly, we have
\[
\begin{split}
&x_1 \argmax(x_0,x_1) \\
&= x_1 \sum_{d=1}^{p-1} d^{-1} (x_0+1) (x_0+2) \cdots (x_0+d) x_1 (x_1-1) \cdots (x_1-(p-d)+1) \\
&= x_1 \sum_{d=2}^{p-2} d^{-1} (x_0+1) (x_0+2) \cdots (x_0+d) x_1 (x_1-1) \cdots (x_1-(p-d)+1) \\
&\quad + (x_0+1) x_1^2 (x_1-1) \cdots (x_1-(p-2)) - (x_0+1) (x_0+2) \cdots (x_0+p-1) x_1^2
\end{split}
\]
and, for the last two terms above, we have
\[
\begin{split}
(x_0+1) x_1^2 (x_1-1) \cdots (x_1-(p-2)) &\equiv -(x_0+1) x_1 (x_1-1) \cdots (x_1-(p-2)) \enspace, \\
(x_0+1) (x_0+2) \cdots (x_0+p-1) x_1^2 &\equiv (p-1)! \cdot \delta_0(x_0) x_1^2 = - (1 - x_0^{p-1}) x_1^2
\end{split}
\]
where we used $x_1^2\equiv x_1((x_i-(p-1))-1)$ and Wilson's Theorem $(p-1)! \equiv -1 \pmod{p}$.

By combining these results to Lemma \ref{lem:relation_max_argmax}, we have
\[
\begin{split}
\max(x_0,x_1)
&= x_0 - x_0 \argmax(x_0,x_1) + x_1 \argmax(x_0,x_1) \\
&\equiv (x_1-x_0) \sum_{d=2}^{p-2} d^{-1} (x_0+1) (x_0+2) \cdots (x_0+d) x_1 (x_1-1) \cdots (x_1-(p-d)+1) \\
&\quad + x_0 - x_0 (x_0+1) x_1 (x_1-1) \cdots (x_1-(p-2)) \\
&\quad - (x_0+1) x_1 (x_1-1) \cdots (x_1-(p-2)) + (1 - x_0^{p-1}) x_1^2 \\
&= (x_1-x_0) \sum_{d=2}^{p-2} d^{-1} (x_0+1) (x_0+2) \cdots (x_0+d) x_1 (x_1-1) \cdots (x_1-(p-d)+1) \\
&\quad + x_0 - (x_0+1)^2 x_1 (x_1-1) \cdots (x_1-(p-2)) + (1 - x_0^{p-1}) x_1^2
\end{split}
\]
and, for the second last term above, we have
\[
(x_0+1)^2 x_1 (x_1-1) \cdots (x_1-(p-2))
\equiv (x_0+1)^2 \cdot (p-1)! \cdot \delta_{p-1}(x_1)
= - (x_0+1)^2 (1 - (x_1+1)^{p-1})
\]
where we used Wilson's Theorem again.
Hence, we have
\[
\begin{split}
\max(x_0,x_1)
&\equiv (x_1-x_0) \sum_{d=2}^{p-2} d^{-1} (x_0+1) (x_0+2) \cdots (x_0+d) x_1 (x_1-1) \cdots (x_1-(p-d)+1) \\
&\quad + x_0 + (x_0+1)^2 (1 - (x_1+1)^{p-1}) + (1 - x_0^{p-1}) x_1^2
\end{split}
\]
which is our claim in the statement.
\end{proof}

\section{Polynomial expressions of some other functions}
\label{sec:other_functions}

In this section, we study the following two $\F_p$-valued functions that are related to $\max$ and $\argmax$:
\begin{eqnarray*}
\ismax(y;x) &=& \chi(\max(x)=y) \enspace, \\
\nummax^{(r)}(x) &=& \# \{x_i\mid \max(x)=x_i)\}^{(r)} \enspace,
\end{eqnarray*}
where $x \in \F_p{}^n$ and $y \in \F_p$.
In practical applications, these functions are useful if there are ``ties'' in the vote.

By a careful interpretation of the definitions, 
we obtain minimal polynomials of these functions (which, however, consist of a lot of terms):
\begin{prop}
Using the notation from \S \ref{sec:notation}, the following are minimal polynomial expressions:
\begin{eqnarray*}
\ismax(y;x)
&=& 
\sum_{t=0}^{p-1} \delta_t(y) \sum_{i=0}^{n-1} \left( \prod_{j<i} L_t(x_j) \cdot \delta_t(x_i) \cdot \prod_{k>i} L_{t+1}(x_k)
\right) \enspace, \\
\nummax^{(0)}(x)&=& \sum_{i=0}^{n-1} \chi(\max(x)=x_i)  \\
&=&  \sum_{i=0}^{n-1}  \sum_{0 \le t \le p-1} \left( 
 \delta_t(x_i) \prod_{j\neq i} L_{t+1} (x_j)
  \right)
  \enspace, \\
  \nummax^{(r)}(x) &=& \sum_{k=1}^{{n}}  k^{(r)} \cdot \chi(\# \{{i}\mid \max(x)=x_i\}=k)  \\
&=&  \sum_{k=1}^{{n}}  k^{(r)}  \left( 
     \sum_{I\in {n\choose k}}
     \sum_{0 \le t \le p-1} \left( 
           \prod_{i\in I}\delta_t(x_i) \prod_{j\not\in I} L_t (x_j)  \right)
  \right) \enspace.
\end{eqnarray*}
\end{prop}
\begin{proof}
For the function $\ismax$, given a constant $t \in \F_p$, we have $\max(x) = t$ if and only if there is an index $i$ satisfying that $x_j < t$ for every $j < i$, $x_i = t$, and $x_k \leq t$ for every $k > i$; such an index $i$ is unique if exists.
This observation (in particular, the uniqueness of $i$) implies our claim.

{For the function $\nummax^{(0)}$, the function value is obtained by first counting the number of indices $i$ with $\max(x) = x_i$ (or equivalently, $\chi(\max(x) = x_i) = 1$) and then taking the remainder of the number modulo $p$ (i.e., just considering the number in $\F_p$).
Moreover, given a constant $t \in \F_p$, we have $\max(x) = x_i = t$ if and only if $x_i = t$ and $x_j \leq t$ for every $j \neq i$.
This observation implies our claim.}

{For the function $\nummax^{(r)}$, given an integer $k \geq 1$ and a constant $t \in \F_p$, we have $\max(x) = t$ and $\# \{i \mid \max(x)=x_i\}=k$ if and only if there is a $k$-element set $I$ of indices satisfying that $x_i = t$ for every $i \in I$ and $x_j < t$ for every $j \not\in I$; such a set $I$ is unique if exists.
This observation (in particular, the uniqueness of $I$) implies our claim (note that $0^{(r)} = 0$ for any $r$).}
\end{proof}

When $p=2$ and $3$, we give the following explicit minimal polynomial expressions of $\ismax(y;x)$:
\begin{prop}
\label{nu:prop:7.2}
When $p=2$, a minimal polynomial expression of $\ismax(y;x)$ is given by
\[
\ismax(y;x)= y+\prod_{i=0}^{n-1}(1+x_i) \enspace.
\]
When $p=3$, a minimal polynomial expression of $\ismax(y;x)$ is given by
\[
\ismax(y;x)= -y^2+y \left( \prod_{i=0}^{n-1}(1+x_i)^2+\prod_{i=0}^{n-1}(1-x_i^2) + 1 \right) + \prod_{i=0}^{n-1} (1 - x_i^2) \enspace.
\]
\end{prop}
\begin{proof}
First, we note that $\ismax(y;x) = 1 - (y - \max(x))^{p-1}$ by the definition of the function.
When $p = 2$, the right-hand side becomes $y + \max(x) + 1$ and now the claim follows from Proposition $\ref{psi-p-2}$.

On the other hand, when $p = 3$, we have
\[
\ismax(y;x)
= 1 - (y - \max(x))^2
= - y^2 - y \max(x) + 1 - \max(x)^2 \enspace.
\]
Now we have $1 - \max(x)^2 = 1$ if $x_i = 0$ for all $i$, and $= 0$ otherwise.
This implies that
\[
1 - \max(x)^2
= \prod_{i=0}^{n-1} \delta_0(x_i)
= \prod_{i=0}^{n-1} (1 - x_i^2)
\]
and now the claim follows from Proposition \ref{psi-p-3}.
\end{proof}

\begin{ex}
When $p=2$, a minimal polynomial expression of $\nummax^{(r)}(x)$ is given by
\[
\nummax^{(r)}(x)= e_{2^r} + n^{(r)} \prod_{i=0}^{n-1}(1-x_i)\enspace.
\]
This can be seen by the following argument.
When $\max(x)=0$, i.e., $x_i=0$ for all $i$, 
we have $\nummax^{(r)}=n^{(r)}$ for any $r$,
which accounts for the second term.
As $(\sum_{i=0}^{n-1}x_i)^{(r)} \equiv e_{2^r}(x) \bmod 2$ by the result of \cite{BPP00} (see also \cite[Example 1]{carry}),
we obtain the equation.
\end{ex}

\section{Future Subject: Multi-digit case}
\label{sec:multi-digit}

We note that the previous sections studied functions with single-digit input values taken from $\F_p$; in such a formulation, to handle larger input values we have to choose a larger prime $p$ as well, which will result in polynomial expressions of the functions with higher degrees and much more involved structures.
Another option to handle larger values is to express the input values in \emph{multi-digit} forms; now each component of the input is identified with its $p$-ary expansion, therefore the entire input is regarded as a two-dimensional matrix over $\F_p$ rather than a one-dimensional vector (over a larger field).
In the latter model, the base field $\F_p$ can be kept small even if the input values become larger.
On the other hand, a large input value will then increase the total number of components of the input matrix, but this shortcoming might sometimes be avoidable in practice by implementation techniques such as parallel computation.
This suggests that polynomial expressions of functions with multi-digit inputs are important as well.

However, even if the polynomial expression of a given function is understood well for single-digit input cases, it is in general a non-trivial task to deduce a polynomial expression of the function for multi-digit input cases.
We leave such multi-digit extensions of the results in this paper as a future research topic, 
and we just conclude this paper with an example:
\begin{prop}
Let $p = 2$, and consider two-bit inputs $y = 2y_1+y_0\in \{0,1,2,3\}$ and $x_i = 2x_{i,1}+x_{i,0}\in \{0,1,2,3\}$ for $0 \leq i \leq n-1$, 
where $y_j,x_{i,j} \in \F_2$.
Then the following is a minimal polynomial expression:
\begin{multline*}
\ismax(y;x) = \ismax(y;x_0,x_1,\dots,x_{n-1}) \\
= y_1y_0
+y_1 \prod_{i=0}^{n-1} (1+x_{i,1}x_{i,0}) 
+(y_1 + y_0)\prod_{i=0}^{n-1}(1+x_{i,1}) 
+(y_1 + 1)\prod_{i=0}^{n-1} (1+x_{i,1})(1+x_{i,0}) \enspace.
\end{multline*}
\end{prop}
\begin{proof}
As the right-hand side of the statement satisfies the minimality conditions for the degrees, it suffices to verify that the values of both terms are equal for any input values.

First we note that, for any set $I$ of index pairs $(i,j)$, we have
\[
\prod_{(i,j) \in I} (1 + x_{i,j}) = \chi(x_{i,j} = 0 \mbox{ for all } (i,j) \in I) \enspace.
\]
Similarly, we have
\[
\prod_{i} (1 + x_{i,1}x_{i,0}) = \chi(\mbox{ for any $i$, either $x_{i,1} = 0$ or $x_{i,0} = 0$ holds }) \enspace.
\]

We divide the argument according to the values of $y_1$ and $y_0$.
When $y_1=y_0=0$, we have $\ismax(y;x) = 1$ if and only if $x_{i,1} = x_{i,0} = 0$ for every index $i$.
Now the right-hand side of the statement becomes $\prod_{i=0}^{n-1} (1 + x_{i,1})(1 + x_{i,0})$, which coincides with $\ismax(y;x)$ by the remark above.

When $y_1=0$ and $y_0=1$, the right-hand side of the statement becomes $\prod_{i=0}^{n-1} (1 + x_{i,1}) + \prod_{i=0}^{n-1} (1 + x_{i,1})(1 + x_{i,0})$.
Now if at least one of $x_{i,1}$ is $1$, then we have $\ismax(y;x) = 0$ by definition, while the value of the polynomial becomes $0$ as well by the remark above, as desired.
In the remaining case where $x_{i,1} = 0$ for every $i$, we have $\ismax(y;x) = 1$ if and only if $x_{i,0} = 1$ for some $i$; while the polynomial now becomes $1 + \prod_{i=0}^{n-1} (1 + x_{i,0})$.
By the remark above, the value of the polynomial coincides with $\ismax(y;x)$, as desired.

When $y_1=1$ and $y_0=0$, the right-hand side of the statement becomes $\prod_{i=0}^{n-1} (1+x_{i,1}x_{i,0}) + \prod_{i=0}^{n-1}(1+x_{i,1})$.
Now if $x_{i,1} = 0$ for every $i$, then we have $\ismax(y;x) = 0$ by definition, while the value of the polynomial becomes $1 + 1 = 0$ as well by the remark above, as desired.
In the remaining case where $x_{i,1} = 1$ for some $i$, let $I$ denote the set of indices $i$ with $x_{i,1} = 1$ (hence now $I \neq \emptyset$).
In this case, we have $\ismax(y;x) = 1$ if and only if $x_{i,0} = 0$ for every $i \in I$; while the polynomial now becomes $\prod_{i \in I} (1 + x_{i,0})$.
By the remark above, the value of the polynomial coincides with $\ismax(y;x)$, as desired.

Finally, when $y_1=y_0=1$, the right-hand side of the statement becomes $1 + \prod_{i=0}^{n-1} (1+x_{i,1}x_{i,0})$.
By the remark above, this polynomial takes the value $1$ if and only if $x_{i,1} = x_{i,0} = 1$ for some index $i$; this condition is precisely the same as the condition for $\ismax(y;x)$ in the present case to take the value $1$, by definition.
This completes the proof.
\end{proof}

 
%

\end{document}